\documentclass[12pt]{amsart}
\usepackage{amsmath}
\usepackage{amsfonts}
\usepackage{amssymb}
\usepackage{graphicx}

\newtheorem{theo}{Theorem}[section]
\theoremstyle{plain}

\newtheorem{cor}[theo]{Corollary}

\newtheorem{example}[theo]{Example}

\newtheorem{lemma}[theo]{Lemma}

\newtheorem{proposition}[theo]{Proposition}
\newtheorem{remark}[theo]{Remark}

\numberwithin{equation}{section}
                    
\begin{document}
\title[Cliquishness and Quasicontinuity of  Two Variables Maps]{Cliquishness and Quasicontinuity 
of  Two
Variables Maps}
\author{A. Bouziad}
\address[]
{D\'epartement de Math\'ematiques,  Universit\'e de Rouen, UMR CNRS 6085, 
Avenue de l'Universit\'e, BP.12, F76801 Saint-\'Etienne-du-Rouvray, France.}
\email{ahmed.bouziad@univ-rouen.fr}
 
\begin{abstract} We
study the existence of continuity points for   mappings 
$f: X\times Y\to Z$  whose  $x$-sections $Y\ni y\to f(x,y)\in  Z$ are
fragmentable  and  $y$-sections $X\ni x\to f(x,y)\in Z$ are 
quasicontinuous, where   $X$ is a Baire space and   $Z$
 is a metric space. For the factor $Y$, we consider two 
 infinite ``point-picking'' games $G_1(y)$ and  $G_2(y)$ defined respectively
for each $y\in Y$  as follows: In the $n$th
inning,  Player {\bf I} gives a dense  set  $D_n\subset Y$, respectively, a dense open set $D_n\subset Y$, then 
Player {\bf II} picks a point $y_n\in D_n$;
{\bf II} wins if $y$ is in the closure of ${\{y_n:n\in\mathbb N\}}$, otherwise
{\bf I} wins. It is shown that
(i) $f$ is
cliquish 
 if {\bf II} has a winning strategy in $G_1(y)$ for every $y\in Y$,  and (ii)  $
f$ is quasicontinuous  if 
  the $x$-sections of  $f$ are continuous   and the set of $y\in Y$
 such that {\bf II} has a winning strategy in $G_2(y)$ is dense in $Y$. Item (i) extends substantially
 a result of Debs (1986) and item (ii)  indicates that 
  the problem of Talagrand (1985) on separately continuous maps has a positive answer for a wide
class of ``small'' compact spaces. 
  
\end{abstract}
\maketitle
\section{introduction} Let $X$ be a topological space and $(Z,d)$
be a metric space. A mapping 
$f:X \to Z$ 
is said to be cliquish \cite{Th} if for any $\varepsilon >0$ and any nonempty open set
$U\subset X$, there is a nonempty open set $O\subset U$ such that
$d(f(x),f(y))<\varepsilon$ for all $x,y\in O$. Following \cite{Ko} (see also \cite{JOPV}), the
mapping $f$ is said to be fragmentable if the restriction of $f$ to each
nonempty subspace of $X$ is cliquish. Fragmentable mappings are said to be of 
the first  class in Debs's paper \cite{D}; given the common  meaning of
``first class functions'', we will adopt
here Koumoullis's
terminology. Recall also that the mapping $f$ is said to be
quasicontinuous \cite{Ke} if for every $\varepsilon>0$, every $x\in X$ and every
neighborhood $V$ of $x$ in $X$, there is a nonempty open set $O\subset V$ such
that $d(f(x),f(y))<\varepsilon$ for each $y\in O$. It is well known (and easily
seen) that  quasicontinuous mappings  are cliquish, and cliquish mappings are
continuous at every point of a residual subset of $X$ (and vice versa if $X$
is a Baire space).\par 
There are in the literature a lot  of studies, dating back at least to Baire
\cite{Ba}, whose purpose is to find conditions (as weak as possible) to insure
the existence of continuity points for mappings of two variables.
Among them there is  the following by Fudali \cite{F}
(1983): Every mapping $f:X\times Y\to Z$, where $X$ is Baire, $Y$ is second 
countable and $Z$ is a metric space,
such that for every $(x,y)\in X\times Y$ the $x$-section $f_x: Y\ni
y\to f(x,y)\in Z$ is cliquish and  the $y$-section $f^y:
X\ni x\to f(x,y)\in Z$ is  quasicontinuous, is a cliquish mapping. 
 See also
\cite{E} for similar results and \cite{N,P} for closely related results
involving quasicontinuous
$x$-sections. There are easy examples showing that  Fudali's
result 
is false for metrizable $Y$ (an example is included here); therefore, 
the following result established  by   Debs in \cite{D} (1986)
is of particular interest: Every mapping $f:X\times Y\to Z$
whose $x$-sections  are fragmentable and   $y$-sections are continuous  is cliquish, provided 
that   $X$ is  a ``special''
Baire space,
$X\times Y$ is a Baire space and $Y$ is   first  countable.  
The interested reader is referred to
\cite{D} 
for the precise assumption on $X$.  \par
In this note 
 Debs's theorem is    improved  
as follows (Corollary 3.5): If the $y$-sections of $f$ are quasicontinuous, the
$x$-sections are fragmentable,
$X$ is Baire 
and the $\pi$-character of each point
of $Y$ is countable, then $f$ is cliquish. 
This  statement is  a special case of one of the  main  results of this paper
(Theorems 3.4 and 3.7) where the problem is considered 
under some fairly general  conditions expressed in terms of
two
  point-picking games played on the factor $Y$ (defined in the next section).
The second main result concerns the mappings $f$ whose $x$-sections are
continuous and
 $y$-section are quasicontinuous; it states, in particular, that such a mapping is quasicontinuous
provided that the space $Y$ has densely many points of countable
$\pi$-character (Corollary 3.8).   \par 
The topic here is closely related to the following problem by
Talagrand \cite{T} (1985): Let $f:X \times Y\to\mathbb R$ be a separately
continuous mapping, where $X$ is a Baire space
and $Y$ is a compact space; is it true that $f$ admits 
 at least a continuity point in $X\times Y$?  The reader is referred
to \cite{BM} for further information about this still-open question. According to Corollary 3.8, we have
 a positive answer if  densely many points of  $Y$ (or $X$) are of
countable $\pi$-character. In view of  the theorem by 
 Juh\'asz and Shelah \cite{JS}, that is, $\pi_\chi(y,Y)\leq t(y,Y)$ for every $y\in
Y$, the answer is also 
  positive  if the compact $Y$  admits  a dense
set of points of countable tightness. The definitions of the cardinal
numbers $\pi_\chi(y,Y)$ and $t(y,Y)$ are recalled below.

\section{Two games}
  
 Let $Y$ be a topological space and 
$\mathcal L$ be a collection of nonempty
subsets
of $Y$. For  
 $y\in Y$, we consider
the following two persons infinite point-picking game 
$G({\mathcal L},y)$ on $Y$. Player I begins and
gives
  $L_0\in\mathcal L $, then Player II chooses a point
$y_0\in L_0$; at stage
$n\geq 1$, Player I chooses   $L_n\in\mathcal L $
and then
Player II gives a point $y_n\in L_{n}$.  A play $(L_n,y_n)_{n\in\mathbb
N}$ is won by Player I   if $y\in\overline{\{y_n:n\in\mathbb
N\}}$; otherwise, II wins.  \par
We will be concerned  in this game  with   two different
collections 
$\mathcal L$ of subsets of $Y$,  namely, the collection ${\mathcal O}(Y)$ of
nonempty open
subsets of $Y$ and the collection ${\mathcal A}(Y)$ of somewhere dense subsets
of $Y$. (When the space $Y$ is clearly identified from the context,  we shall
simply write  ${\mathcal O}$ and $\mathcal A$.) Recall that a subset
$F\subset Y$ is said to be somewhere dense in $Y$ if the interior of its
closure ${\rm Int}({\overline F})$ in $Y$ is
nonempty. It should be mentioned that if  $\mathcal L$ is the collection of all neighborhoods
of $y$ in $Y$, then $G({\mathcal L}, y)$ is  the game introduced by 
Gruenhage
in \cite{G}. The game $G({\mathcal O},y)$ is 
the pointwise version of   the one introduced by Berner and  Juh\'asz
in their paper \cite{BJ} (from which the term ``picking-point
game''
is taken):  In the $n$th step, Player I gives a nonempty open set $U_n\subset Y$, then
II picks a point $y_n\in U_n$; I wins if $\{y_n:n\in\mathbb N\}$ is dense in
$Y$.  \par 

Following the terminology of \cite{S2}, the {\it dual game} $G^*({\mathcal
A},y)$ of $G({\mathcal A},y)$ (respectively, $G^*({\mathcal
O},y)$  of $G({\mathcal O},y)$) on $Y$
is defined
as follows: At stage $n$, Player I gives a dense open set $D_n\subset
Y$ (respectively, a dense  set $D_n\subset Y$) and then Player II chooses
$y_n\in D_n$. Player
II wins if $y\in\overline{\{y_n:n\in\mathbb N\}}$. Using
the  techniques of \cite{S1}
one can show that these games are indeed dual, meaning  that Player II has
a
winning strategy  in the game $G^*({\mathcal A},y)$ (respectively, Player I
has a winning strategy in $G^*({\mathcal A},y)$) if and only if Player
I  has a winning strategy  in $G({\mathcal
A},y)$ (respectively, Player II
has a winning strategy in $G({\mathcal A},y)$). This is also
true for the class ${\mathcal O}$.  \par 
 The next  statement  and the discussion after its proof show that the
difference between the games 
 $G({\mathcal O},y)$ and $G({\mathcal A},y)$ is significant.

\begin{proposition}  Let $y\in Y$ and ${\mathcal N}$ be a collection
of  subsets of $Y$ such that
\begin{itemize}
\item[{\rm (i)}] for every neighborhood $U$ of $y$ in $Y$ there is a finite
collection ${\mathcal F}\subset\mathcal N$ such that ${\rm Int}(\cap{\mathcal
F})\not=\emptyset$
and $\cap{\mathcal F}\subset U$;
\item[{\rm (ii)}] 
the closure of the set $A=\{z\in Y:|\{N\in {\mathcal N}:z\not\in N\}|\leq\aleph_0\}$ is a neighborhood of
$y$ in  
$Y$.
\end{itemize}
Then  Player I
has a winning strategy in the game $G({\mathcal A},y)$.
\end{proposition}
\begin{proof}  Let us 
fix a bijective map ${\mathbb N}\ni n\to (\phi(n),\psi(n))\in {\mathbb N}\times \mathbb N$
such that $n>\phi(n)$ for every $n\geq 1$. Put $\tau_y(\emptyset)=A$. For $y_0\in A$
(i.e., the answer  of
Player II) let ${\mathcal S}_0=\{S_n^0:n\in\mathbb N\}$ be an enumeration of the collection  of all
sets of the form $\cap\mathcal F$,
where $\mathcal F$ is
a finite
subcollection  of ${\mathcal N}_0=\{N\in{\mathcal N}: y_0\not\in N\}$  (we adopt
the convention $\cap\emptyset=Y$). Define
$\tau_y(y_0)=S_0^0\cap A$ if $S_0^0\cap A\in\mathcal A$
and $\tau_y(y_0)=A$ otherwise.  At stage $n\geq 1$, to define $\tau_y(y_0,\ldots,y_n)$ 
put ${\mathcal N}_{n}=\{N\in{\mathcal N}:y_{n}\not\in N\}$ and
denote by ${\mathcal S}_{n}=\{S_{k}^{n}:k\in\mathbb N\}$  the collection of all intersections $\cap\mathcal F$, where $\mathcal F$
is a finite subcollection of  $\cup_{i\leq n}{\mathcal N}_i$. Then,  put
$\tau_y(y_0,\ldots,y_{n})=A\cap S_{\psi(n)}^{\phi(n)}$ if $A\cap S_{\psi(n)}^{\phi(n)}\in\mathcal A$
and $\tau_y(y_0,\ldots,y_{n})=A$ otherwise. The definition of $\tau_y$ is
complete.\par
To show that $\tau_y$ is a winning strategy,  let $(y_n)_{n\in\mathbb N}\subset Y$ be a
play which is compatible with $\tau_y$ and let 
  $U\subset \overline A$ be a neighborhood
 of $y$ in $Y$. There is a finite set ${\mathcal F}\subset \mathcal N$ such that
$\cap{\mathcal F}\subset U$ and ${\rm Int}(\cap{\mathcal F})\not=\emptyset$.
 Let ${\mathcal F}_1={\mathcal F}\cap(\cup_{n\in\mathbb N}{\mathcal N}_n)$. 
 We assume that ${\mathcal F}_1\not=\emptyset$ (otherwise,
$\{y_n:n\in\mathbb N\}\subset U$). Put
 $S=\cap{\mathcal F}_1$ and choose  $n\in\mathbb N$
such that $S=S_{\psi(n)}^{\phi(n)}$;  since
  $\emptyset\not={\rm Int}(\cap{\mathcal F})\subset \overline
A$, the set
 $S\cap A$ belongs to $\mathcal A$. It follows that $y_{n+1}\in S$, hence $y_{n+1}\in U$ since
  $y_{n+1}\in N$ for every $N\in{\mathcal F}\setminus{\mathcal F}_1$.
 \end{proof} 
Recall that a network at $y$ in $Y$ is a collection $\mathcal N$ of 
subsets of $Y$ such that   every
neighborhood  of $y$ in $Y$ contains some nonempty  member of $\mathcal N$. A $\pi$-base at $y$ in $Y$
is a network at $y$, all members of which are open.
The space $Y$ is said to have a countable $\pi$-character at $y\in Y$, in
symbol $\pi_\chi(y,Y)\leq\aleph_0$, if $y$ has a countable $\pi$-base in $Y$.\par
Proposition 2.1 applies in the case $\pi_\chi(y,Y)\leq \aleph_0$ as well as in many other cases. To illustrate this, let us consider for a cardinal number $\kappa$ the Cantor cube $2^\kappa$
of weight $\kappa$. It is well known that  $\pi_\chi(y,2^\kappa)=\kappa$ for
every $y\in2^\kappa$ (see \cite{H});
since $2^\kappa$ is a regular space, it follows that if $\kappa$ is
uncountable then 
${\mathcal A}(2^\kappa)$ does not include any countable $\pi$-network at any point
of $2^\kappa$. Also, if $\kappa$ is uncountable, then Player II
has a winning strategy in the game $G({\mathcal O},y)$ for every $y\in 2^\kappa$:
It suffices to confront Player II in the dual game $G^*({\mathcal O},y)$ to
the dense subset of $2^\kappa$ given by $\Sigma({\overline y})=\{z\in
2^\kappa:|\{i\in \kappa: z(i)\not={\overline y}(i)\}|\leq\aleph_0\}$, where
${\overline y}=1-y$.  However, Player
I has always a winning strategy in the games $G({\mathcal A},y)$, $y\in
2^\kappa$. Indeed, for
$y\in 2^\kappa$, the collection ${\mathcal N}=\{\{z\in 2^\kappa:z(i)=y(i)\}:i\in\kappa\}$
satisfies the assumption of Proposition 2.1. Observe also that Player I has a winning strategy
in the games $G({\mathcal O},y)$ played on the dense subspace $\Sigma(0)$ of
$2^\kappa$, for every $y\in\Sigma(0)$ \cite{G}.
\par

 Clearly, if Player I has a winning strategy in the game $G({\mathcal L},y)$ on
$Y$,
then the 
collection $\mathcal L$ is a $\pi$-network at $y$ in $Y$. (Of course, this  holds  even if Player II does not 
have a winning strategy in this game.)
The next lemma, needed below,
gives us a bit more. Let $Y^{<\omega}$ stand for the set
of finite sequences in $Y$.
\begin{lemma} Suppose that   Player I has a winning strategy
$\tau$ in the game $G({\mathcal L},y)$.
Then, for every neighborhood $V$ of $y$ in $Y$, Player I has a winning strategy
$\sigma$  in the game $G({\mathcal L},y)$ such that $\sigma\subset V$, that
is, $\sigma(s)\subset V$
for every $s\in Y^{<\omega}$.
\end{lemma}
\begin{proof} Fix some $L_0\in\mathcal L$ such that $L_0\subset V$. For
every
finite sequence $s=(y_0,\ldots,y_n)\in Y^{<\omega}$ such that
$y\in\overline{\{y_0,\ldots,y_n\}}$ (no separation axiom is assumed), put $\sigma(s)=L_0$. For the remaining
sequences in $Y^{<\omega}$, including the empty sequence (that is, the first 
move of Player I), we proceed as follows.  Let
$s\in Y^{k}$ be such a sequence ($k\in\mathbb N$). If
$\tau(s)\subset V$, put $\sigma(s)=\tau(s)$ and $t_s=\emptyset$.
If not, write $s=(y_0,\ldots,y_k)$ (if $s\not=\emptyset$) and choose a finite
sequence $t_s=(x_0^{s},\ldots,x_{n_{s}}^s)\in
Y^{n_s}$ such that the sequence
$(s,t_s)$ is compatible with $\tau$, $\{x_0^{s},\ldots,x_{n_{s}}^s\}\cap V=\emptyset$ and
$\tau(s,t_s)\subset V\setminus\overline{\{y_0,\ldots,y_k\}}$
(or $\tau(s,t_s)\subset V$ if $s=\emptyset$); such a
sequence
exists since $\tau$ is a winning strategy.
Then define $\sigma(s)=\tau(s,t_s)$. The definition of $\sigma$ is complete.\par
Let $(y_n)_{n\in\mathbb N}\subset Y$ be a sequence which is compatible with
$\sigma$ and let $W\subset V$ be a neighborhood of $y$
in $Y$. We may suppose that $y\not\in\cup_{n\in\mathbb N}\overline{\{y_0,\ldots,y_n\}}$.
Put $s_n=(y_0,\ldots,y_n)$ ($n\in\mathbb
N$); then, the sequence $(z_n)_{n\in\mathbb N}$ starting with $t_\emptyset$
and obtained 
from $(y_n)_{n\in\mathbb N}$ by inserting 
each  $t_{s_n}$
just after $y_n$,  is compatible with $\tau$. Hence there
is $p\in\mathbb N$ such that $z_p\in W$;  since   no term of the
sequences
$t_\emptyset$ and $t_{s_n}$ ($n\in\mathbb N$)  belongs to $V$, $z_p\in\{y_n:n\in\mathbb
N\}$. 
  \end{proof}
\section{Main results}
The main results rest on the
following proposition. In its proof we shall make use of the description
of first category sets in term of the Banach-Mazur game. For
 a space $X$ and  $R\subset  X$, a play in the game  $BM(R)$ (on $X$) 
is a sequence  $(V_n,U_n)_{n\in\mathbb N}$ of pairs of nonempty open subsets
of $X$ produced  alternately by two players $\beta$
and $\alpha$ as  follows:  $\beta$ is the first to move and gives $V_0$, then
Player $\alpha$ gives $U_0\subset V_0$; at stage $n\geq 1$,
the open set $V_n\subset U_n$ being chosen by $\beta$, Player $\alpha$ gives
$U_n\subset
V_n$. 
Player $\alpha$ wins the play if
 $\cap_{n\in\mathbb N}U_n\subset R$. It is well known that
$X$ is $BM(R)$-$\alpha$-favorable (i.e., $\alpha$
has a winning strategy in the game $BM(R)$)  if and only if  $R$
is a residual subset of $X$. The reader is referred to  \cite{oxt}. \par
  Let us say that the space $Y$ is {\it fragmented} by $\Delta\subset Y\times Y$
  if every nonempty subspace of  $Y$ admits a nonempty (relatively) open subset $U$
  such that $U\times U\subset \Delta$. In the next
  statement, $X$, $Y$ are topological spaces,  
  $(\Delta_x)_{x\in X}$ is an $X$-indexed collection of subsets of $Y\times Y$ 
and $\mathcal L$ is  collection of nonempty
subsets of $Y$  such that for every $y\in Y$, Player I has a winning strategy in the game $G({\mathcal L},y)$.
\begin{proposition} Let $R$ be 
 a second category  subset  of $X$
 such that $Y$ is fragmented by $\Delta_x$ for each $x\in R$. Then, for
  every nonempty open set $V\subset Y$, there exist a nonempty
open set $U\subset X$, $y\in V$ and
$L\in\mathcal L$, with $L\subset V$,
such that for every   $b\in L$  the set $\{x\in U: (b,y)\in \Delta_x\}$ is 
dense in $U$.
\end{proposition} 
\begin{proof} Assume, on the contrary, that the claim
is false for 
some nonempty open set $V\subset Y$, and let us prove    
that $X\setminus R$ is a residual subset of $X$ (i.e., $R$
is of the first category in $X$). For each $y\in V$ and
$L\in\mathcal L $, with $L\subset V$, let $D(y,L)$
be the set of $x\in X$ for which there is $b\in L$ such that $(b,y)\not\in \Delta_a$
for every $a$ in some neighborhood of $x$  in $X$; by our assumption,
the open set  $D(y,L)$
is dense in $X$. \par 
For each  $(x,y,L)$ (where $(x,y)\in X\times V$, $L\in\mathcal L$ and $L\subset
V$) such that
$x\in D(y,L)$, choose  a point
$c_L(x,y)\in L$ and an open neighborhood $O(L,x,y)$ of $x$ in $X$ such
that  $(c_L(x,y),y)\not\in \Delta_a$ for every
$a\in O(L,x,y)$. Let us fix for each $y\in V$ a winning
strategy $\tau_y\subset V$ 
for Player I in the game $G({\mathcal L},y)$ (Lemma 2.2). We shall define a 
winning
strategy
$\sigma$ for Player $\alpha$ in the Banach-Mazur game $BM(X\setminus R)$ on $X$
as follows. Let $V_0$ be the first move of Player $\beta$ in the game
$BM(X\setminus R)$ and put
$\sigma(V_0)=V_0\cap D_0(y_0,L_0^{y_0})$, where $y_0$ is an arbitrary (but
fixed) point of $V$ and $L_0^{y_0}=\tau_{y_0}(\emptyset)$. Define 
$F_0=\{y_0\}$.\par
At stage $1$, if
$V_1$ is
the response of $\beta$ to $\sigma(V_0)$, first choose
$x_1\in V_1$, put
$y_1=c_{L_0^{y_0}}(x_1,y_0)$ and $F_1=\{y_1\}$. Then define
$\sigma(V_0,V_1)$ to be the nonempty open subset of $X$  given by
$$V_1\cap O(L_0^{y_0},x_1,y_0)\cap D(y_0,L_1^{y_0})\cap
D(y_1,L_0^{y_1}),$$
where $L_1^{y_0}=\tau_{y_0}(y_1)$ and $L_0^{y_1}=\tau_{y_1}(\emptyset)$.\par 
 At stage $2$, if
$V_2$ is
the response of $\beta$ to $\sigma(V_0, V_1)$, first choose
$x_2\in V_2$,   
 put $y_2=c_{L_1^{y_0}}(x_2,y_0)$, $y_3=c_{L_0^{y_1}}(x_2,y_1)$ and
$F_2=\{y_2,y_3\}$; then define
$\sigma(V_0,V_1,V_2)$ to be the nonempty open set given by
$$V_2\cap O(L_1^{y_0},x_2,y_0)\cap O(L_0^{y_1},x_2,y_1)\cap D(y_0,
L_2^{y_0})\cap D(y_1,L_1^{y_1})\cap\big[ \bigcap_{y\in F_2}D(y,L_0^y)\big],$$
where $L_2^{y_0}=\tau_{y_0}(y_1,y_2)$, $L_1^{y_1}=\tau_{y_1}(y_3)$ and
$L_0^{y}=\tau_{y}(\emptyset)$ for $y\in F_2$. \par 
Continuing  inductively, the notations will become more and more
complicated but the process allows us to define
a strategy $\sigma$ for Player $\alpha$ in the
game $BM(X\setminus R)$ with the following property:  To each play
$s=(V_n)_{n\in\mathbb N}$ for
$\beta$ against $\sigma$ corresponds a set
$F_s=\cup_{n\in\mathbb N}F_n\subset Y$
such that
for each $y\in F_s$  there is a play $(y_n)_{n\in\mathbb N}\subset F_s$
of Player II in the game
$G({\mathcal L},y)$  against the strategy
$\tau_y$  such that $(y_n,y)\not\in \Delta_x$ for every $x\in\cap_{n\in\mathbb N}V_n$ and
$n\in\mathbb N$. \par 
To conclude, let   $s=(V_n)_{n\in\mathbb N}$ be a play for Player $\beta$ against $\sigma$ and
let us
show that $\cap_{n\in\mathbb N}V_n\subset X\setminus R$.  
Let $x\in\cap_{n\in\mathbb N}V_n$ and  suppose that  $x\in R$. There is 
an open set $O\subset Y$ such that $O\cap F_s\not=\emptyset$ and $(O\times
O)\cap (F_s\times F_s)\subset \Delta_x$.
Let $y\in O\cap F_s$; since $y\in\overline{\{y_n:n\in\mathbb N\}}$, there is $n
\in\mathbb N$ such that $(y_n,y)\in\Delta_x$, which is 
 a contradiction.
 \end{proof}
Throughout the rest of the paper,   $f: X\times Y\to Z$ is a mapping, where
$(Z,d)$ is a metric space. Let $\varepsilon>0$. We shall
apply Proposition 3.1 to
the collection of subsets of $Y\times Y$
of the form $\Delta_x=\{(y,z)\in Y\times Y: d(f(x,y),f(x,z))<\varepsilon\}$, $x\in X$.
Clearly, 
the ``$\varepsilon$-fragmentability'' of  the mapping $f_x$  for $x\in X$ as
defined in the introduction means
that $Y$ is fragmented by $\Delta_x$.
\begin{remark} {\rm Let $y\in Y$ be such that $f_x$ is continuous at $y$
for every
$x\in R$ (notations of Proposition 3.1).\par
(i) If Player I  has
a winning strategy $\tau_y$
in the game $G({\mathcal L},y)$ on $Y$, 
then, involving only the strategy $\tau_y$, the same method in the above
proof allows to establish the following assertion:  $(*)$ For every
$\varepsilon>0$ and any
neighborhood
$V$ of $y\in Y$, there is a nonempty open set $U\subset X$ and $L\in\mathcal L$, $L\subset V$,
such  the sets $\{x\in U:
d(f(x,b),f(x,y))<\varepsilon\}$, $b\in L$, are
dense in $U$.\par 
(ii) If $y$ has a countable network ${\mathcal N}\subset\mathcal L$,
then the above property $(*)$ can be proved easily as follows: Proceeding by
contradiction as in the proof of Proposition 3.1,
since the open subset
$D(y,L)$ of $X$ is   dense in $X$ for every $L\in\mathcal L$ with $L\subset V$, there is $x\in R$ such that $x\in D(y,L)$ for every
$L\in\mathcal N$ with $L\subset V$. This gives
a countable set $\{y_n:n\in\mathbb N\}\subset Y$ such that
$y\in\overline{\{y_n:n\in\mathbb N\}}$ and $d(f(x,y_n),f(x,y))\geq  \varepsilon$
for every $n\in\mathbb N$, which is absurd since $f_x$ is continuous at $y$.} 
\end{remark}
The following interesting concept
 is formulated in \cite{MN} (concepts quite similar  were studied
by K. B\" ogel in his papers \cite {B1,B2}): The mapping $f:X\times Y\to Z$ is said to be
horizontally continuous at $(a,b)\in X\times Y$
if for every neighborhood $W$ of $f(a,b)$ in $Z$ and every neighborhood $U\times
V$ of
$(a,b)$ in $X\times Y$, there are a nonempty open set $O\subset U$
and $y\in V$ such that $f(O\times\{y\})\subset W$. 
\par The mapping $f:X\times Y\to Z$ is said to be lower quasicontinuous
with respect to the variable $x$ at the point $(a,b)\in X\times Y$ if for every
neighborhood $W$ of $f(a,b)$ in $Z$ and every neighborhood $U\times V$ of $(a,b)$ in $X\times Y$, there
is a nonempty open set $O\subset U$ such that for each $x\in O$ there is
$y\in V$ such that  $f(x,y)\in W$. (This concept is introduced in a forthcoming
paper with J.-P. Troallic.)  Lower quasicontinuity with respect to the
variable
$y$ is defined similarly.  Note that the quasicontinuity of $f^b$ at $a\in X$
implies
that $f$ is horizontally quasicontinuous at $(a,b)$ which in turn implies
that $f$ is lower quasicontinuous with respect to the variable $x$ at $(a,b)$. \par
We continue to assume  (in Proposition 3.3 and Theorem 3.4 below) 
that for every $y\in Y$, Player I has a winning strategy in the game $G({\mathcal L},y)$. 

\begin{proposition} Suppose that $f_x$ is fragmentable for each $x$ in a second
category
set $R$ in $X$, $f^y$ is cliquish for every $y\in Y$ and $f$ is
horizontally quasicontinuous. Let $V\subset Y$ be a
nonempty open set. Then, there exist $b\in Y$ and a nonempty open set
$O\times W\subset X\times V$  
such that $d(f(x,y),f(x',b))\leq \varepsilon$ for every
$x, x'\in O$ and $y\in W$, in each of the following:
\begin{itemize}
\item[{\rm (i)}] ${\mathcal L}={\mathcal O}$.
\item[{\rm (ii)}] ${\mathcal L}={\mathcal A}$,  $f^y$ is quasicontinuous for every $y\in Y$
and $f$ is lower  quasicontinuous with
respect to the variable $y$.
\end{itemize}
\end{proposition}
\begin{proof}  By Proposition 3.1,  there are 
$b\in V$,
a nonempty open set $U\subset X$ and $L\in\mathcal L$ with $L\subset V$, 
such that for every $y\in  L$ the set $D_y=\{x\in U:
d(f(x,y),f(x,b))\leq\varepsilon/2\}$ is dense in $U$.  Since $f^b$
is cliquish in both
cases, there is a nonempty open
set $O\subset U$ such that 
${\rm diam}(f^b(O))\leq\varepsilon/2$. \par 
To prove (i), we take $W=L$.  Suppose that
$d(f(x_0,y_0),f(x_1,b))>\varepsilon$ for some $x_0,x_1\in O$ and $y_0\in L$.
Since $f$
is horizontally quasicontinuous at $(x_0,y_0)$ and $L$ is open, there is a nonempty open set
$O_1\subset O$ and $y_1\in L $
such that $d(f(a,y_1),f(x_1,b))>\varepsilon$ for every $a\in O_1$. Let $a_1\in
O_1\cap D_{y_1}$; it follows from 
 $d(f(a_1,y_1),f(a_1,b))\leq\varepsilon/2$ that
$d(f(a_1,b),f(x_1,b))> \varepsilon/2$,
which is a contradiction.\par   
 To prove (ii), we take $W={\rm Int}({\overline L}\cap V)$ ; since
 $L\in\mathcal A$ and 
$L\subset V$, $W$
is nonempty.  Suppose
that $d(f(x_0,y_0),f(x_1,y))> \varepsilon$ for some $x_0,x_1\in O$ and $y_0\in W$. 
There is a nonempty open set $W_1\subset W$
such that for each $y\in W_1$ there is $a\in O$ such that
$d(f(a,y),f(x_1,b))>\varepsilon$; taking
$y_1\in W_1\cap L$, we obtain $d(f(a_1,y_1),f(x_1,b))>\varepsilon$ for some
$a_1\in O$. Since $f^{y_1}$ is quasicontinuous, there is a nonempty open set $O_1\subset O$
such that $d(f(a,y_1),f(x_1,b))>\varepsilon$ for every $a\in O_1$; as in the proof of (i), taking
$a\in O_1\cap D_{y_1}$ gives the contradiction $d(f(a,b),f(x_1,b))>\varepsilon/2$. 
\end{proof}

Now we state the first main result of this note. 
\begin{theo} Suppose that $f$ is horizontally quasicontinuous, $f^y$ is cliquish
for every $y\in Y$ and  $f_x$ is fragmentable  for each $x$ in the Baire space $X$.   Then $f$ is
cliquish in each of the following:
\begin{itemize}
\item[{\rm (i)}] ${\mathcal L}={\mathcal O}$;
\item[{\rm (ii)}] ${\mathcal L}={\mathcal A}$, $f^y$ is quasicontinuous for
every $y\in Y$ and $f$ is lower  quasicontinuous with
respect to the variable $y$.
\end{itemize}
\end{theo}
\begin{proof} Let $U\times V\subset X\times Y$ be a nonempty open set and
$\varepsilon>0$. By Proposition 3.3 (for $X=R=U$),  there
are
$b\in Y$ and a nonempty open set $O\times W\subset U\times V$
such that $d(f(x,y),f(x',b))\leq\varepsilon$ for every $x,x'\in O$ and $y\in W$.
For every $x,x'\in O$ and $y,y'\in W$, we have 
\begin{align*}
d(f(x,y),f(x',y'))&\leq d(f(x,y),(x',b))+d(f(x',b),f(x',y'))\\
&\leq 2\varepsilon.
\end{align*}
\end{proof}
In view of Proposition 2.1 and Theorem 3.4-(i), we obtain the following improvement
of Debs's result  mentioned in the introduction.
\begin{cor} Suppose that for each $y\in Y$, $\pi_{\chi}(y, Y)\leq\aleph_0$ and $f^y$ is quasicontinuous, and
  $X$ is a Baire space and $f_x$ is fragmentable for each $x\in X$.
Then
$f$ is cliquish.
\end{cor}

The concept of cliquish mapping  extends in a natural way to 
 mappings taking their values in   uniform spaces. Therefore, Theorem
3.4 and Corollary 3.5 hold more generally for every uniform  space $Z$. Let us
also note that the assumption on the $y$-sections of $f$ in Theorem 3.4 allows 
to assume
in this statement that the $x$-sections are fragmentable for $x$ belonging
to  a dense Baire subspace of $X$. \par 
 If the $x$-sections of the mapping $f$ are continuous, then 
the cliquishness of $f$ in Theorem 3.4 can be significantly improved, as we propose
to show in what follows. We need  a variant of Proposition 3.3.
\begin{proposition} Suppose that  $f_x$ is continuous for
each $x\in X$, $f^y$
is quasicontinuous for every $y\in Y$ and  $X$ is a Baire space. Let
$b\in
Y$ be such that Player I has a winning strategy in the game $G({\mathcal A},b)$
on
$Y$ and let $V$ be a neighborhood of $b$ in $Y$. Then, for every nonempty
open set $U\subset X$, there is a nonempty open set $O\times W\subset U\times V$
such that $d(f(x,y),f(x,b))\leq\varepsilon$ for every $(x,y)\in O\times W$.
\end{proposition}
\begin{proof} By Remark 3.2-(i), there are $L\in \mathcal A$, with $L\subset V$,
and a nonempty
open set $O\subset U$ such that for every $y\in
L$ the set $\{x\in U:d(f(x,y),f(x,b))\leq\varepsilon/2\}$ is dense in $U$. It remains to follow   the  proof of Proposition
3.3-(ii) (more simply, because here the $f_x$'s are  continuous). 
\end{proof}
The following is a variant of Theorem 3.4-(ii) (the assumption on the $x$-sections of $f$ 
is strengthened, but there are fewer constraints on $Y$ and the conclusion is stronger).
\begin{theo} Suppose that for each point $y$ in a dense subset of $Y$,
 Player I has a winning strategy in the game $G({\mathcal A},y)$. If 
$X$ is a Baire space,
$f_x$  is continuous for every $x\in X$ and $f^y$
is quasicontinuous for every $y\in Y$, then $f$ is quasicontinuous.
\end{theo}
\begin{proof} Let $(a,b)\in X\times Y$, $U\times V$  be a neighborhood of
$(a,b)\in X\times Y$
and $\varepsilon>0$. We may suppose that $d(f(a,y),f(a,b))<\varepsilon$ for
every $y\in V$.
Let $c\in V$ be such that  Player I has a winning strategy in the game
$G({\mathcal A},c)$. Since $f^c$ is quasicontinuous,
there is a nonempty open set $O_1\subset U$ such that
$d(f(x,c),f(a,c))<\varepsilon$ for every $x\in O_1$. Let $O_2\times W\subset
O_1\times V$ be a nonempty open set such that $d(f(x,y),f(x,c))<\varepsilon$
for every
$(x,y)\in O_2\times W$ (Proposition 3.6). For every $(x,y)\in O_2\times W$, we have 
\begin{align*}
d(f(x,y),f(a,b))\leq &\ 
d(f(x,y),f(x,c))+d(f(x,c),f(a,c))\\
&\ +d(f(a,c),f(a,b))\\
\leq &\ 3\varepsilon.
\end{align*} 
\end{proof}
The following is a consequence of Theorem 3.7; it can be also obtained (directly
and more simply) by using Remark 3.2-(ii).
\begin{cor} Suppose  that for
every $y$
in a dense subset of $Y$, the collection $\mathcal A$
includes a countable network at $y$ in $Y$. If  $X$ is a Baire space, $f^y$
is quasicontinuous 
 for each $y\in Y$ and $f_x$ is continuous for each $x\in X$,
then
$f$ is quasicontinuous.
\end{cor} 
Corollary 3.8 shows that the question of  Talagrand \cite{T} mentioned in the
introduction has a positive answer if one of the two spaces $X$ and $Y$ has 
a dense subset of points of   countable $\pi$-character (since  $Y$
is compact, the product
$X\times Y$ is
Baire hence the quasicontinuous mapping $f: X\times Y\to Z$ has at least a continuity point). Related to this, let us recall from the theorem  of
\v Sapirovski\v\i\,
\cite{Sh1} that all compacts spaces $Y$ that cannot be continuously mapped onto the
Tychonoff
 cube $[0,1]^{\omega_1}$ satisfy the conditions of Corollary 3.8. This is also the case
 of all hereditarily normal compact spaces, by another result of \v Sapirovski\v\i\, \cite{Sh2,Sh3}.
 Taking into account 
 the theorem of  Juh\'asz and Shelah \cite{JS} that
  $\pi_\chi(y,Y)\leq t(y,Y)$ for every $y$ in the compact space $Y$, the
answer to Talagrand's question is  also positive if  $Y$ has a dense set of points  
of countable tightness. The tightness
$t(y,Y)$ of $y$ in $Y$  is the smallest
cardinal $\kappa$ such that whenever $y\in\overline A$, $A\subset Y$,
there is a set $B\subset A$ with $|B|\leq \kappa$ such that $y\in\overline
B$.\par 
\begin{example} {\rm In conclusion, we return to the question raised in the
introduction whether it
is 
possible to assume
in the main results that the $x$-sections  of the mappings $f: X\times Y\to
Z$
are only cliquish, as  is the case     
if the factor $Y$ is (locally) second countable \cite{F}.  Unfortunately,  this is not possible  even for 
metrizable $Y$. To show this, let $Y$ be a metrizable space such that the interior of
every  separable subspace of $Y$ is empty and take  $X$ to be the countably compact (hence Baire) subspace
$X=\Sigma(0)$ of the Cantor
space $2^Y$. Then, the  evaluation mapping $X\times Y\ni (x,y)\to x(y)\in \{0,1\}$
 is cliquish in the variable $y$, 
continuous
in the variable $x$  but not cliquish
if  $Y$ is dense in itself. Furthermore, inverting the roles of
$X$
and $Y$, and taking $Y$ to be a Baire space (e.g., completely metrizable), we
obtain an
example showing that the assumption concerning the
$y$-sections
of the mapping $f$ in 
 Theorems 3.4 and 3.7 cannot be replaced by  
 the lower quasicontinuity  of $f$ with respect to the variable $x$.}
 \end{example}


\end{document}